\documentclass[12pt,reqno]{article}

\usepackage[usenames]{color}
\usepackage{amssymb}
\usepackage{amsmath}
\usepackage{amsthm}
\usepackage{amsfonts}
\usepackage{amscd}
\usepackage{graphicx}

\usepackage[colorlinks=true,
linkcolor=webgreen,
filecolor=webbrown,
citecolor=webgreen]{hyperref}

\definecolor{webgreen}{rgb}{0,.5,0}
\definecolor{webbrown}{rgb}{.6,0,0}

\usepackage{color}
\usepackage{fullpage}
\usepackage{float}

\usepackage{graphics}
\usepackage{latexsym}
\usepackage{epsf}
\usepackage{breakurl}

\newcommand{\seqnum}[1]{\href{https://oeis.org/#1}{\rm \underline{#1}}}
\def\modd#1 #2{#1\ \mbox{\rm (mod}\ #2\mbox{\rm )}}

\begin{document}

\theoremstyle{plain}
\newtheorem{theorem}{Theorem}
\newtheorem{corollary}[theorem]{Corollary}
\newtheorem{lemma}[theorem]{Lemma}
\newtheorem{proposition}[theorem]{Proposition}
\newtheorem{question}[theorem]{Question}

\theoremstyle{definition}
\newtheorem{definition}[theorem]{Definition}
\newtheorem{example}[theorem]{Example}
\newtheorem{conjecture}[theorem]{Conjecture}

\theoremstyle{remark}
\newtheorem{remark}[theorem]{Remark}

\begin{center}
\vskip 1cm{\LARGE
Derivatives and Integrals of Polynomials Associated with Integer Partitions
}
\vskip 1cm
Madeline Locus Dawsey, Tyler Russell, and Dannie Urban\\
Department of Mathematics\\
University of Texas at Tyler\\
Tyler, TX 75799\\
USA\\
\href{mailto:mdawsey@uttyler.edu}{\tt mdawsey@uttyler.edu}  \\
\href{mailto:trussell12@patriots.uttyler.edu}{\tt trussell12@patriots.uttyler.edu}  \\
\href{mailto:durban@patriots.uttyler.edu}{\tt durban@patriots.uttyler.edu}  \\
\end{center}

\vskip 0.2in

\begin{abstract} 
Integer partitions express the different ways that a positive integer may be written as a sum of positive integers.  Here we explore the analytic properties of a new polynomial $f_\lambda(x)$ that we call the partition polynomial for the partition $\lambda$, with the aim to learn new properties of partitions.  We prove a recursive formula for the derivatives of $f_\lambda(x)$ involving Stirling numbers of the second kind, show that the set of integrals from 0 to 1 of a normalized version of $f_\lambda(x)$ is dense in $[0,1/2]$, pose a few open questions, and formulate a conjecture relating the integral to the length of the partition.  We also provide specific examples throughout to support our speculation that an in-depth analysis of partition polynomials could further strengthen our understanding of partitions.
\end{abstract}

\section{Introduction and statement of results}

A \emph{partition} $\lambda$ of a nonnegative integer $n$ is a non-increasing sequence $(\lambda_1,\lambda_2,\lambda_3,\dots)$ of positive integers, called the \emph{parts} of the partition, which sum to $n$.  The sum $n$ is called the \emph{size} of the partition.  By convention, the only partition of size $n=0$ is the empty partition with no parts.  Although partitions are classical objects and simple to define, they are also complex additive structures with elusive analytic properties and are still a significant topic of study in number theory today \cite{A,AE,HW}.  Hardy, Rademacher, Ramanujan, and many other well known number theorists of the twentieth century made major breakthroughs in our understanding of partitions via the generating function of $p(n)$, the \emph{partition function} \cite[\seqnum{A000041}]{OEIS}, which counts the number of partitions of size $n$.  These results include the famous Hardy-Ramanujan asymptotic formula which was proven analytically using the circle method \cite{HR}: $$p(n)\sim\frac{1}{4n\sqrt{3}}e^{\pi\sqrt{\frac{2n}{3}}}\hspace{.25cm}\text{as }n\rightarrow\infty;$$ and the Ramanujan congruences which can be proven using the theory of modular forms \cite{R1,R3}: $$p(5n+4)\equiv\modd{0} {5},\hspace{.2cm}p(7n+5)\equiv\modd{0} {7},\hspace{.2cm}\text{and}\hspace{.2cm}p(11n+6)\equiv\modd{0} {11}$$  for all integers $n\geq0$.  Rademacher's main contribution was an exact formula for $p(n)$ as an absolutely convergent infinite series involving Kloosterman sums and the $I$-Bessel function \cite{Rademacher}.  Most of the discoveries that have been made about partitions so far have resulted from studying $p(n)$.

Recent work by the first author, Just, and Schneider \cite{DJS} takes a different approach to studying partitions.  They define a map from the set of all partitions to the set of natural numbers called the \emph{supernorm} of a partition, previously called the Heinz number: \cite[\seqnum{A305078}]{OEIS}, which sheds more light on properties of partitions by relating them to the multiplicative structure of prime factorizations of integers.  For notational convenience, their work utilizes frequency notation $\lambda=\langle1^{m_1},2^{m_2},\dots,k^{m_k}\rangle$, where $m_i$ is the multiplicity of the part $i$ in the partition $\lambda$ and $k$ is the largest part of $\lambda$, as opposed to the traditional additive notation.  The \emph{supernorm} of a partition $\lambda=\langle1^{m_1},2^{m_2},\dots,k^{m_k}\rangle$ is defined by $$\widehat{N}(\lambda)=\prod_{i=1}^kp_i^{m_i},$$ where $p_i$ denotes the $i$th prime.  In an attempt to further understand the parts of a partition $\lambda$ and their multiplicities, the \emph{length} $\ell(\lambda)$ (the number of parts), the size $|\lambda|$, the \emph{norm} $N(\lambda)$ (the product of the parts), and the supernorm $\widehat{N}(\lambda)$, among other partition statistics, Just \cite{J} defined a polynomial to which one can apply calculus and analysis in order to analyze properties of partitions.  To define this new polynomial, let $\lambda=\langle1^{m_1},2^{m_2},\dots,k^{m_k}\rangle$ be a partition written in frequency notation.  We define the \emph{partition polynomial} $f_\lambda$ by 
\begin{equation}\label{def1}
f_\lambda(x)=\sum_{i=1}^km_ix^i.
\end{equation}

Differentiation and integration of the partition polynomial both have the potential to reveal certain properties of partitions.  The most basic observations one can make in the direction of differentiation are that $f_\lambda(1)=\ell(\lambda)$ and $f'_\lambda(1)=|\lambda|$.  It is natural to then ask what the higher derivatives of $f_\lambda$ tell us about $\lambda$ when we evaluate at $x=1$.  Does $f''_\lambda(1)$ give some generalization of the size or the length of $\lambda$?  We first determine what the higher derivatives look like in general, and then we can make some progress toward answering this question by calculating a few specific examples which lead to some open questions.

\begin{theorem}\label{derivativethm}
Given a partition $\lambda=\langle1^{m_1},2^{m_2},\dots,k^{m_k}\rangle$, the $d$th derivative of its partition polynomial is
\begin{equation}\label{derivativeformula}
f_{\lambda}^{(d)}(x)=\sum_{i=1}^{k}i^{d}m_ix^{i-d}-\sum_{j=0}^{d-1}{d\brace j}x^{j-d}f_{\lambda}^{(j)}(x)
\end{equation}
for all $0\leq d\leq k$, where ${d\brace j}$ is the Stirling number of the second kind. 
\end{theorem}

The $d=0$ derivative $f_\lambda^{(0)}$ denotes the partition polynomial itself.  Note that the second sum is empty when $d=0$ and that $f_\lambda^{(d)}(x)=0$ for all $d>k$.  For background on Stirling numbers of the second kind, see \cite{Stanley} and \cite[\seqnum{A008277}]{OEIS}, for example.  We prove Theorem \ref{derivativethm} in Section \ref{proofofderivativethm} and provide further discussion, examples, and some open questions about derivatives of partition polynomials in Section \ref{derivativequestions}.

In addition to studying derivatives, one may also attempt to understand partitions further by studying integrals.  Before taking integrals of partition polynomials, we first normalize them and restrict their domain as follows in order to more easily compare partitions.  The \emph{normalized partition polynomial} $\hat{f}_\lambda : [0, 1] \to \mathbb{R}$ of the partition $\lambda=\langle1^{m_1},2^{m_2},\dots,k^{m_k}\rangle$ is defined by
		\begin{equation}\label{def2}
		\hat{f}_\lambda(x) = \frac{1}{\ell(\lambda)} \sum_{i=1}^{k} m_i x^i.
		\end{equation}
		
	\begin{figure}[h]
		\centering
		\includegraphics[scale=0.55]{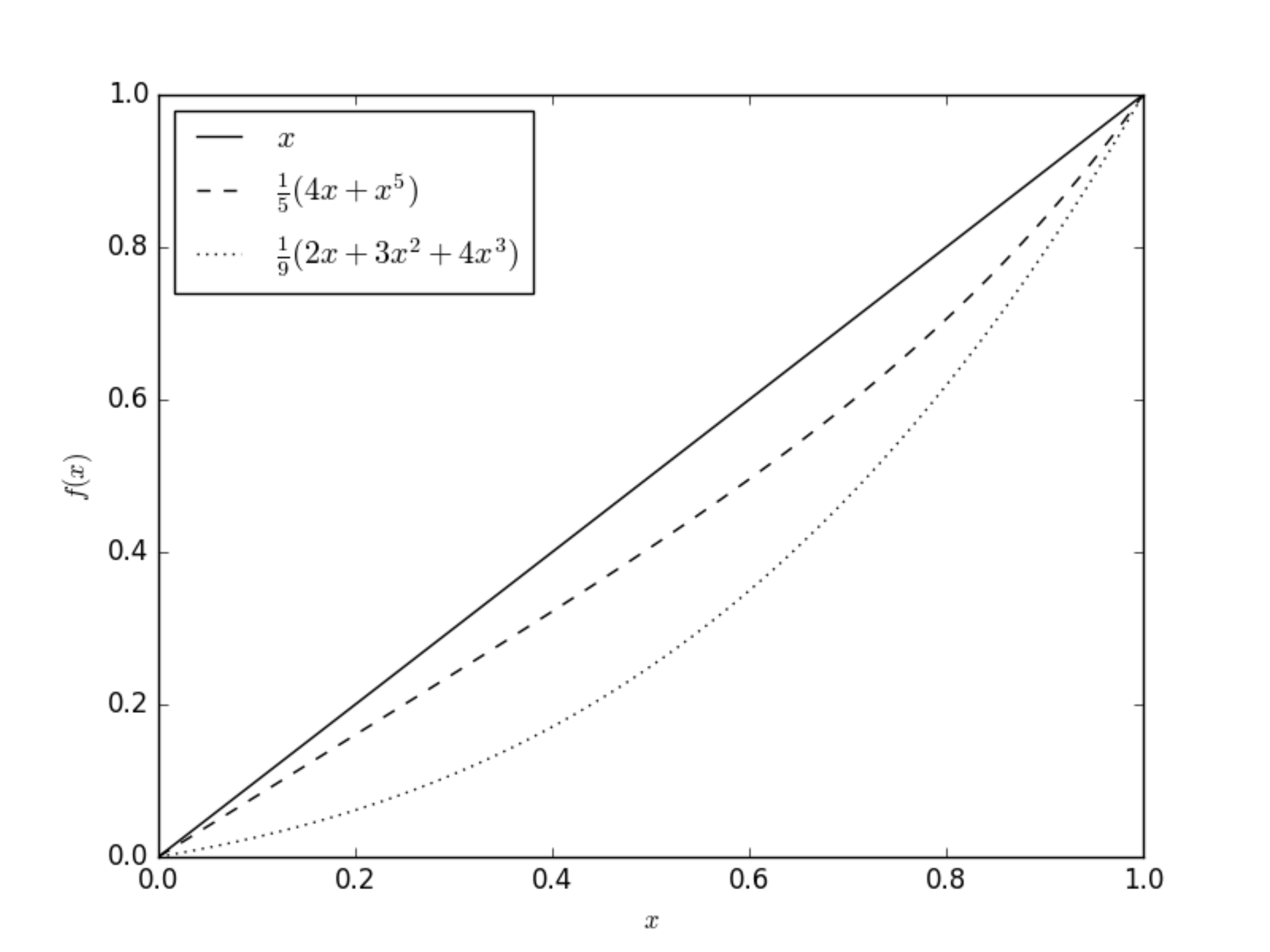}
		\caption{Some normalized partition polynomials.}
		\label{figure1}
	\end{figure}
	
	It is straightforward to see that the integral of a normalized partition polynomial must have a value between 0 and 1/2 as illustrated in Figure \ref{figure1} (see Section \ref{section_integrals} for more details).  Knowing this, it is natural to ask which values in the interval $[0, 1/2]$ can be the integral of a normalized partition polynomial.  The following theorem shows that one can find normalized partition polynomials whose integrals take any value in the interval $[0, 1/2]$.
	
	\begin{theorem}
		\label{main_result}
		The integrals of normalized partition polynomials are dense in $[0,1/2]$.
	\end{theorem}
	
	We provide more basic results on integrals of normalized partition polynomials in Section \ref{section_integrals}, and we prove Theorem \ref{main_result} in Section \ref{sec_integralproofs}.  The remaining question is what the value of the integral tells us about the partition, and this question is still open, for the most part.
	
	\begin{question}\label{question3}
	Is there a combinatorial interpretation of ${\displaystyle\int_0^1}\hat{f}_\lambda(x)\,dx$?
	\end{question}
	
	Question \ref{question3} leads to many other related open questions.  For example, it is natural to study integrals of all normalized partition polynomials for partitions of fixed size $n$, but it is still an open problem to obtain a readily accessible combinatorial interpretation of these integrals.  This turns out to be difficult, but one could obtain results about partitions of large size $n$ which rely on asymptotic results.   It would be desirable to prove a result for all partition sizes $n$ that does not rely on asymptotics.
	
	One could also ask whether certain types of partitions of fixed size $n$ generally yield larger or smaller integrals.  Some computation indicates that although fewer parts do not always correspond to smaller integrals, this trend may hold on average.

\begin{conjecture}\label{conj}
    Let $n \in \mathbb{N}$, and let $\mathrm{Avg}(n, \ell)$ denote the average of all integrals of normalized partition polynomials over all partitions of size $n$ into $\ell$ parts. Then we expect that
    \begin{equation*}
        \mathrm{Avg}(n, 1) \leq \mathrm{Avg}(n, 2) \leq \cdots \leq \mathrm{Avg}(n, n) .
    \end{equation*}
\end{conjecture}

Numerical computation confirms that Conjecture \ref{conj} holds for $n \leq 50$.  We partially prove the conjecture by proving the first inequality $\mathrm{Avg}(n, 1) \leq \mathrm{Avg}(n, 2)$ for all $n$ and the second inequality $\mathrm{Avg}(n, 2) \leq \mathrm{Avg}(n, 3)$ in an asymptotic sense for large values of $n$.

\begin{theorem}\label{avg1}
    Let $n \in \mathbb{N}$. Then $\mathrm{Avg}(n, 1) \leq \mathrm{Avg}(n, 2)$.
\end{theorem}

Note that $\mathrm{Avg}(n, \ell)$ may be computed by averaging all of the integrals for partitions of size $n$ with $\ell$ parts, but alternatively it may be computed by evaluating the integral of the sum of all partitions of size $n$ with $\ell$ parts (see Section \ref{sec_operations} for more details).  Instead of considering each individual integral, we can evaluate a single integral involving all of the parts.  Although this method uses and produces the same information, it simplifies the estimation. The larger, combined partition has length $\ell$ times the number of partitions of size $n$ into $\ell$ parts.

The proof of Theorem \ref{avg1} requires one to know the exact form of all partitions of size $n$ and length $\ell$ when $\ell=1$ and $\ell=2$. This is a difficult problem in general, so we instead rely on asymptotic approximations for the second inequality. For this reason, the result is only guaranteed for sufficiently large $n$. 

\begin{theorem}\label{avg2}
    For sufficiently large $n$, $\mathrm{Avg}(n, 2) \leq \mathrm{Avg}(n, 3)$.
\end{theorem}

Proving Conjecture \ref{conj} in general seems difficult. More progress could possibly be made in an asymptotic sense using the same method as the proof of Theorem \ref{avg2}. To fully generalize these results, one may need to find a general formula for the number of parts of size $i$ among all partitions of size $n$, or among all partitions of size $n$ and length $\ell$.
	
We prove Theorems \ref{avg1} and \ref{avg2} in Section \ref{sec_averages}.

\section{Proof of Theorem \ref{derivativethm}}\label{proofofderivativethm}
Let $\lambda=\langle 1^{m_1},2^{m_2},\dots,k^{m_k}\rangle$.  We will prove the derivative formula \eqref{derivativeformula} for $f_\lambda$ by induction.  First note that when $d=0$, \eqref{derivativeformula} is trivially true.  When $d=1$, we have that
$$f_\lambda^{(1)}(x)=\sum_{i=1}^{k}im_ix^{i-1},$$
which satisfies \eqref{derivativeformula} since ${d\brace 0}=0$.  Next, when $d=2$, we have that
\begin{align*}
f_{\lambda}^{(2)}(x)&=\sum_{i=2}^{k}i(i-1)m_ix^{i-2}\\
 &=\sum_{i=2}^{k}i^2m_ix^{i-2}-\sum_{i=2}^{k}im_ix^{i-2}.\\
 \intertext{Adding and subtracting $1m_1x^{-1}$, we have}
 f_{\lambda}^{(2)}(x)&=\sum_{i=2}^{k}i^2m_ix^{i-2}-\sum_{i=1}^{k}im_ix^{i-2}+1m_1x^{-1},\\
 \intertext{which yields the following when $1m_1x^{-1}$ is absorbed into the first sum:}
 f_{\lambda}^{(2)}(x) &=\sum_{i=1}^{k}i^2m_ix^{i-2}-x^{-1}\sum_{i=1}^{k}im_ix^{i-1}.
 \end{align*}
 The second sum above contains the first derivative $f^{(1)}_\lambda(x)$, so we have that
 \begin{align*}
 f_\lambda^{(2)}(x)&=\sum_{i=1}^{k}i^2m_ix^{i-2}-x^{-1}f_{\lambda}^{(1)}(x),
\end{align*}
which also satisfies \eqref{derivativeformula}.  Now, let $N\in\mathbb{N}$ and suppose \eqref{derivativeformula} holds for $d= N$.  We rewrite the $N$th derivative as follows:
\begin{align}\label{inductionstep}
f_{\lambda}^{(N)}(x)=\sum_{i=N}^{k}i^{N}m_ix^{i-N}+\sum_{i=1}^{N-1}i^Nm_ix^{-(N-i)}-\sum_{j=0}^{N-1}{N\brace j}x^{j-N}f_{\lambda}^{(j)}(x)\>.
\end{align}
We take the derivative to obtain
\begin{align}\label{intermediate}
f_{\lambda}^{(N+1)}(x)&=\sum_{i=N+1}^{k}i^{N}(i-N)m_ix^{i-N-1}-\sum_{i=1}^{N-1}i^N(N-i)m_ix^{-(N-i)-1}\nonumber\\
 &\hspace{15mm}-\sum_{j=0}^{N-1}{N\brace j}(j-N)x^{j-N-1}f_{\lambda}^{(j)}(x)-\sum_{j=0}^{N-1}{N\brace j}x^{j-N}f_{\lambda}^{(j+1)}(x).
\end{align} 
We use \eqref{inductionstep} to rewrite the first sum in \eqref{intermediate} as
\begin{align}\label{firstsum}
\sum_{i=N+1}^{k}i^{N}(i-N)m_ix^{i-N-1}
 &=\sum_{i=N+1}^{k}i^{N+1}m_ix^{i-(N+1)}-Nx^{-1}f_{\lambda}^{(N)}(x)+N\sum_{i=1}^{N-1}i^Nm_ix^{-(N-i)-1}\nonumber\\
 &\hspace{20mm}-N\sum_{j=0}^{N-1}{N\brace j}x^{j-N-1}f_{\lambda}^{(j)}(x)+N^{N+1}m_Nx^{-1} \> .
\end{align}
Distributing in the second and third sums of \eqref{intermediate}, we have that
\begin{align}\label{secondsum}
-\sum_{i=1}^{N-1}i^N(N-i)m_ix^{-(N-i)-1}=-N\sum_{i=1}^{N-1}i^Nm_ix^{-N+i-1}+\sum_{i=1}^{N-1}i^{N+1}m_ix^{-N+i-1}
\end{align}
and
\begin{align}\label{thirdsum}
-\sum_{j=0}^{N-1}{N\brace j}(j-N)x^{j-N-1}f_{\lambda}^{(j)}(x)&=N\sum_{j=0}^{N-1}{N\brace j}x^{-N+j-1}f_\lambda^{(j)}(x)\nonumber\\
&\hspace{2cm}-\sum_{j=0}^{N-1}j{N\brace j}x^{-N+j-1}f_\lambda^{(j)}(x).
\end{align}
Shifting the indices of the fourth sum of \eqref{intermediate}, we have that
\begin{align}\label{fourthsum}
-\sum_{j=0}^{N-1}{N\brace j}x^{j-N}f_{\lambda}^{(j+1)}(x)=-\sum_{j=1}^{(N-1)+1}{N\brace j-1}x^{-N+j-1}f_\lambda^{(j)}(x).
\end{align}
Putting together \eqref{firstsum}, \eqref{secondsum}, \eqref{thirdsum}, and \eqref{fourthsum}, rearranging, and using the fact that ${N\brace N}=1$, we then have that
\begin{align*}
f_{\lambda}^{(N+1)}(x)&=\sum_{i=N+1}^{k}i^{N+1}m_ix^{i-(N+1)}+\left(N^{N+1}m_Nx^{-1}+\sum_{i=1}^{N-1}i^{N+1}m_ix^{-(N+1-i)}\right)\\
 &\hspace{10mm}+\left(-N{N\brace N}x^{-1}f_{\lambda}^{(N)}(x)-\sum_{j=0}^{N-1}j{N\brace j}x^{j-N-1}f_{\lambda}^{(j)}(x)\right)\\
 &\hspace{10mm}+\Bigg(N\sum_{i=1}^{N-1}i^Nm_ix^{-(N+1-i)}-N\sum_{j=0}^{N-1}{N\brace j}x^{j-N-1}f_{\lambda}^{(j)}(x)\\
 &\hspace{10mm}-N\sum_{i=1}^{N-1}i^Nm_ix^{-(N+1-i)}+N\sum_{j=0}^{N-1}{N\brace j}x^{j-N-1}f_{\lambda}^{(j)}(x)\Bigg)\\
 &\hspace{10mm}-\sum_{j=1}^{(N+1)-1}{N\brace j-1}x^{j-N-1}f_{\lambda}^{(j)}(x)\>.
 \end{align*}
Combining the terms in the grouped expressions above, we obtain
\begin{align}\label{almostfinished}
 f_{\lambda}^{\>(N+1)}(x)&=\sum_{i=N+1}^{k}i^{N+1}m_ix^{i-(N+1)}+\sum_{i=1}^{(N+1)-1}i^{N+1}m_ix^{i-(N+1)}\nonumber\\
 &\hspace{10mm}-\sum_{j=0}^{(N+1)-1}j{N\brace j}x^{j-N-1}f_{\lambda}^{(j)}(x)-\sum_{j=1}^{(N+1)-1}{N\brace j-1}x^{j-N-1}f_{\lambda}^{(j)}(x).
 \end{align}
 Notice that the third sum in \eqref{almostfinished} can be rewritten as the sum from $j=1$ to $(N+1)-1$, since the $j=0$ term is zero.  We combine the first two sums and the last two sums as follows:
 \begin{align*}
 f_{\lambda}^{\>(N+1)}(x)&=\sum_{i=1}^{k}i^{N+1}m_ix^{i-(N+1)}-\sum_{j=1}^{(N+1)-1}\left(j{N\brace j}+{N\brace j-1}\right)x^{j-N-1}f_{\lambda}^{(j)}(x)\\
 &=\sum_{i=1}^{k}i^{N+1}m_ix^{i-(N+1)}-\sum_{j=1}^{(N+1)-1}{N+1\brace j}x^{j-(N+1)}f_{\lambda}^{(j)}(x),
\end{align*}
where this last equality uses a classical recurrence relation for Stirling numbers of the second kind \cite{Stanley}.  Notice again that the last sum can be rewritten as the sum from $j=0$ to $(N+1)-1$, since the $j=0$ term includes ${N+1\brace 0}=0$.  This completes the proof of Theorem \ref{derivativethm}.

\section{Derivatives of partition polynomials}\label{derivativequestions}

There are many open questions that arise when studying the derivatives of partition polynomials.  It is clear that $f_\lambda^{(d)}(1)\geq0$ for all $d\geq0$, since all of the parts of partitions are positive and all of their multiplicities are nonnegative.  It is also clear that evaluating the first few derivatives at $x=1$ yields
\begin{align*}
f_\lambda^{(0)}(1)&=\ell(\lambda),\\
f_\lambda^{(1)}(1)&=|\lambda|,\\
f_\lambda^{(2)}(1)&=\sum_{i=1}^ki^2m_i-{2\brace1}|\lambda|,\\
f_\lambda^{(3)}(1)&=\sum_{i=1}^ki^3m_i-{3\brace2}\sum_{i=1}^ki^2m_i-\left({3\brace1}-{3\brace2}{2\brace1}\right)|\lambda|,\\
f_\lambda^{(4)}(1)&=\sum_{i=1}^ki^4m_i-{4\brace3}\sum_{i=1}^ki^3m_i-\left({4\brace2}-{4\brace3}{3\brace2}\right)\sum_{i=1}^ki^2m_i\\
&\hspace{1.875cm}-\left({4\brace1}-{4\brace2}{2\brace1}-{4\brace3}{3\brace1}+{4\brace3}{3\brace2}{2\brace1}\right)|\lambda|.
\end{align*}
It would be nice to find a combinatorial interpretation of the $d$th derivative of the partition polynomial at $x=1$.  As a result, properties of the derivatives could potentially provide new information about the partition.  A good start would be to find an explicit formula for $f_\lambda^{(d)}(1)$ for all $0\leq d\leq k$ that is more enlightening than simply plugging $x=1$ into \eqref{derivativeformula}.  It may be helpful to consider \emph{$k$th moments of partitions}, defined by Zemel \cite{Z} as $$p_k(\lambda):=\sum_{i\geq1}i^km_i,$$ as these moments appear in the higher derivatives of $f_\lambda$.

\begin{question}
Is there a nice combinatorial interpretation of $f_\lambda^{(d)}(1)$?
\end{question}

We now provide an explicit example of the derivatives of the partition polynomials for two specific partitions of the same length and size.
\begin{example}\label{example}
Let $\lambda_1=\langle1^1,2^2,3^0,4^0,5^1\rangle$ and $\lambda_2=\langle1^1,2^1,3^1,4^1\rangle$.  Here are the partition polynomial, its derivatives, and their evaluations at $x=1$ for $\lambda_1$:\\
\begin{minipage}{.5\textwidth}
\begin{align*}
f_{\lambda_1}^{(0)}(x)&=x^5+2x^2+x \\
f_{\lambda_1}^{(1)}(x)&=5x^4+4x+1 \\
f_{\lambda_1}^{(2)}(x)&=20x^3+4 \\
f_{\lambda_1}^{(3)}(x)&=60x^2 \\
f_{\lambda_1}^{(4)}(x)&=120x \\
f_{\lambda_1}^{(5)}(x)&=120
\end{align*}
\end{minipage}%
\begin{minipage}{.35\textwidth}
\begin{align*}
f_{\lambda_1}^{(0)}(1)&=4 \\
f_{\lambda_1}^{(1)}(1)&=10 \\
f_{\lambda_1}^{(2)}(1)&=24 \\
f_{\lambda_1}^{(3)}(1)&=60 \\
f_{\lambda_1}^{(4)}(1)&=120 \\
f_{\lambda_1}^{(5)}(1)&=120
\end{align*}
\end{minipage}\\

\noindent Here are the partition polynomial, its derivatives, and their evaluations at $x=1$ for $\lambda_2$:\\
\begin{minipage}{.565\textwidth}
\begin{align*}
f_{\lambda_2}^{(0)}(x)&=x^4+x^3+x^2+x \\
f_{\lambda_2}^{(1)}(x)&=4x^3+3x^2+2x+1 \\
f_{\lambda_2}^{(2)}(x)&=12x^2+6x+2 \\
f_{\lambda_2}^{(3)}(x)&=24x+6 \\
f_{\lambda_2}^{(4)}(x)&=24
\end{align*}
\end{minipage}%
\begin{minipage}{.21\textwidth}
\begin{align*}
f_{\lambda_2}^{(0)}(1)&=4 \\
f_{\lambda_2}^{(1)}(1)&=10 \\
f_{\lambda_2}^{(2)}(1)&=20 \\
f_{\lambda_2}^{(3)}(1)&=30 \\
f_{\lambda_2}^{(4)}(1)&=24
\end{align*}
\end{minipage}\\

\noindent Although $\lambda_1$ and $\lambda_2$ share the same length $\ell(\lambda_1)=\ell(\lambda_2)=4$ and the same size $|\lambda_1|=|\lambda_2|=10$, the evaluations of their partition polynomial derivatives at $x=1$ yield different numbers starting at the second derivative.
\end{example}

Obviously the evaluation of the partition polynomial at $x=1$ can distinguish any two partitions of different lengths, and the evaluation of its first derivative can distinguish any two partitions of the same length and different sizes.  A natural open question that arises from Example \ref{example} is whether the evaluation of its second derivative, or third derivative, or $d$th derivative for some positive integer $d$, can distinguish any two unequal partitions, even if they have the same length and size.  Answering this question could possibly shed more light on a combinatorial interpretation of the higher derivatives, or vice versa.  For the following question, let $\mathrm{lg}(\lambda)$ denote the largest part of the partition $\lambda$, and recall that $f_\lambda^{(d)}(x)=0$ for all $d>\mathrm{lg}(\lambda)$.

\begin{question}\label{question2}
If $\lambda,\lambda'$ are any two unequal partitions, is it true that $f_\lambda^{(d)}(1)\neq f_{\lambda'}^{(d)}(1)$ for some positive integer $d\leq\min\{\mathrm{lg}(\lambda),\mathrm{lg}(\lambda')\}$?
\end{question}

One specific direction in which to investigate Question \ref{question2} is to search for a counterexample for a fixed positive integer $d$: a pair of unequal partitions $\lambda,\lambda'$ for which $f_\lambda^{(r)}(1)=f_{\lambda'}^{(r)}(1)$ for all $r\leq d$.  We provide a partial answer to this question: the second derivative evaluation does not necessarily distinguish two unequal partitions of the same length and size.

\begin{proposition}
There exists a positive integer $N$ such that if $n\geq N$, then there are two unequal partitions $\lambda_1,\lambda_2$ with $\ell(\lambda_1)=\ell(\lambda_2)$ and $|\lambda_1|=|\lambda_2|=n$ such that $f_{\lambda_1}^{(2)}(1)=f_{\lambda_2}^{(2)}(1)$.
\end{proposition}

\begin{proof}
We use the pigeonhole principle to prove the proposition.  Denote the number of partitions of $n$ into $s$ parts by $p(n,s)$.  It can be shown \cite{KK} that $$p(n,s)\sim\frac{n^{s-1}}{s((s-1)!)^2}$$ as $n\rightarrow\infty$ for $s=O(1)$.  In particular, we have that $p(n,5)\sim n^4/2880$.  It suffices to show that $f_\lambda^{(2)}(1)$ is bounded by an asymptotically smaller expression.  Suppose $\lambda=\langle1^{m_1},2^{m_2},\dots,k^{m_k}\rangle$ is a partition of size $n$ and length 5.  Then we have that $$f_\lambda^{(2)}(1)=\sum_{i=1}^ki^2m_i-n\leq k^2n-n\leq n^3-n.$$  Since $(n^3-n)/p(n,5)\rightarrow0$ as $n\rightarrow\infty$, there must be two unequal partitions of the same size and length 5 with equal second derivative evaluations.
\end{proof}

In fact, this argument holds for all $s\geq5$, so the second derivative evaluation fails to distinguish unequal partitions of the same size and length at least 5.  One could also refine the question to ask how many derivatives are necessary in order to distinguish any two unequal partitions of fixed size $n$ and length $s$.

Another interesting observation is that taking the derivative of the partition polynomial of one partition yields the partition polynomial of a different partition.  Explicitly, if we define the partition $\lambda=\lambda^{(0)}=\langle1^{m_1},2^{m_2},\dots,k^{m_k}\rangle$, then $f_\lambda^{(1)}(x)$ is the partition polynomial of the new partition $\lambda^{(1)}$ defined by $\lambda^{(1)}=\langle1^{2m_2},2^{3m_3},\dots,(k-1)^{km_k}\rangle$.  Continuing in this way, we obtain the following (finite) sequence of partitions $\left(\lambda^{(d)}\right)_{0\leq d<k}$ whose partition polynomials are related by differentiation:
\begin{align}\label{sequence}
\lambda^{(d)}&=\left\langle1^{(d+1)!m_{d+1}/1!},2^{(d+2)!m_{d+2}/2!},\dots,(k-d)^{k!m_k/(k-d)!}\right\rangle\hspace{.25cm}\text{for all }0\leq d<k.
\end{align}
Strictly speaking, these partitions also have parts of size zero which we exclude, and parts of negative size which occur with multiplicity zero.  We have that for all $0\leq d<k$,
\begin{equation}\label{lengthsandsizes}
\ell\left(\lambda^{(d)}\right)=\sum_{i=d+1}^k\frac{i!}{(i-d)!}m_i\hspace{1cm}\text{and}\hspace{1cm}\left|\lambda^{(d)}\right|=\sum_{i=d+1}^k\frac{i!}{(i-d-1)!}m_i.
\end{equation}
Therefore, the process of differentiation guarantees that the following relationship between the sizes and lengths of the partitions in the sequence \eqref{sequence} holds for all $1\leq d\leq k$:
\begin{equation}\label{relationship}
\left|\lambda^{(d-1)}\right|=\ell\left(\lambda^{(d)}\right)+d!m_d,
\end{equation}
where $\ell\left(\lambda^{(k)}\right)=0$ since $\lambda^{(k)}$ is the empty partition.

\begin{example}
Let $\lambda=\langle1^1,2^0,3^3,4^1\rangle$.  Then we have
\begin{align*}
\lambda^{(0)}&=\langle1^1,2^0,3^3,4^1\rangle,\\
\lambda^{(1)}&=\langle1^0,2^9,3^4\rangle,\\
\lambda^{(2)}&=\langle1^{18},2^{12}\rangle,\\
\lambda^{(3)}&=\langle1^{24}\rangle.
\end{align*}
We also see that $\ell\left(\lambda^{(0)}\right)=5$, $\left|\lambda^{(3)}\right|=4!m_4=24$, and
\begin{align*}
\left|\lambda^{(0)}\right|=\ell\left(\lambda^{(1)}\right)+1!m_1&=14,\\
\left|\lambda^{(1)}\right|=\ell\left(\lambda^{(2)}\right)+2!m_2&=30,\\
\left|\lambda^{(2)}\right|=\ell\left(\lambda^{(3)}\right)+3!m_3&=42.
\end{align*}
\end{example}

It would be interesting to investigate other relationships among partitions in the sequence \eqref{sequence} which are related by differentiation of partition polynomials, such as patterns in the sequence of lengths or the sequence of sizes from \eqref{lengthsandsizes} or \eqref{relationship}, the zeros of the polynomials, relationships among the polynomials associated with other sequences of partitions, etc.

\section{Integrals of normalized partition polynomials}\label{section_integrals}

As Figure \ref{figure1} suggests, there are a few properties that all normalized partition polynomials have in common. Namely, they all start at the point $(0, 0)$, end at the point $(1, 1)$, and are bounded above by the line $y=x$.  In the following proposition, recall that normalized partition polynomials are defined as functions from $[0,1]$ to $\mathbb{R}$.
	
	\begin{proposition}
		\label{prop2.2}
		Let $\hat{f}_\lambda(x)$ be a normalized partition polynomial. Then following properties hold.
		\begin{enumerate}
			\item $\hat{f}_\lambda(0) = 0$ and $f_\lambda(1) = 1$.
			\item $\hat{f}_\lambda(x) \leq x$.
		\end{enumerate}
	\end{proposition}
	
	\begin{proof}
	We omit the proof of part 1, and we prove part 2 here.  Since $0\leq x\leq 1$, we have that $x^i\leq x$ for all $i\geq1$.  Then we see that $$\hat{f}_\lambda(x)=\frac{1}{\ell(\lambda)}\sum_{i=1}^km_ix^i\leq\frac{1}{\ell(\lambda)}\sum_{i=1}^km_ix.$$  The observation that $\sum_{i=1}^km_i=\ell(\lambda)$ completes the proof.
	\end{proof}
	
	Additionally, we can characterize completely the cases when $\hat{f}_\lambda$ has more than two fixed points.
	
	\begin{proposition}
		\label{prop2.3}
		Let $\hat{f}_\lambda$ be a normalized partition polynomial. Then $\hat{f}_\lambda(x) = x$ if and only if $m_1 = \ell(\lambda)$; that is, $\lambda$ only contains parts of size $1$.
	\end{proposition}
	
	\begin{proof}
	Suppose $m_1=\ell(\lambda)$.  Then $\hat{f}_\lambda(x)=\frac{1}{\ell(\lambda)}\cdot\ell(\lambda)x=x$.  Now suppose $m_j>0$ for some $j>1$.  Then, using the fact that $x^j<x$ when $0<x<1$, we have that $$\hat{f}_\lambda(x)=\frac{1}{\ell(\lambda)}\sum_{i=1}^km_ix^i<\frac{1}{\ell(\lambda)}\sum_{i=1}^km_ix=x.$$  This completes the proof.
	\end{proof}

	\subsection{Integration basics}
	
	In this subsection, we show some basic integration results for normalized partition polynomials. We first show that integration on the interval $[0, 1]$ may be viewed as a finite sum of ``harmonic-like" numbers.
	
	\begin{theorem}
		\label{integral_formula}
		Let $\hat{f}_\lambda$ be a normalized partition polynomial. Then we have that
		\begin{equation*}
		\int_{0}^{1} \hat{f}_\lambda(x) \,dx = \frac{1}{\ell(\lambda)} \sum_{i=1}^{k} \frac{m_i}{i+1} .
		\end{equation*}
	\end{theorem}
	
	\begin{proof}
		We integrate directly to obtain
		\begin{align*}
		\int_{0}^{1}\hat{f}_\lambda(x)\,dx & = \int_{0}^{1} \frac{1}{\ell(\lambda)} \sum_{i=1}^{k} m_i x^i\,dx = \frac{1}{\ell(\lambda)} \sum_{i=1}^{k} m_i\cdot\frac{x^{i+1}}{i+1} \bigg|_{0}^{1} = \frac{1}{\ell(\lambda)} \sum_{i=1}^{k} \frac{m_i}{i+1}.
		\end{align*}
	\end{proof}
	
	We now obtain an elementary bound for integrals of normalized partition polynomials.
	
	\begin{proposition}\label{inequal}
		Let $\hat{f}_\lambda$ be a normalized partition polynomial. Then
		\begin{equation*}
		0 < \int_{0}^{1}\hat{f}_\lambda(x)\,dx \leq\frac{1}{2} .
		\end{equation*}
	\end{proposition}
	
	Proposition \ref{inequal} follows directly from Proposition \ref{prop2.2}.
	
	Proposition \ref{prop2.3} and the discussion afterwards show that the right inequality in Proposition \ref{inequal} is strict if a partition contains parts of size greater than $1$. Despite this, one can find a sequence of normalized partition polynomials for partitions with parts greater than $1$ whose integrals approach $1/2$. We discuss this further in Section \ref{sec_integralproofs}.
	
	To make the distinction between partitions with parts greater than $1$ and those which only have parts of size $1$, we call the former \emph{non-trivial} partitions.

	\subsection{Operations on partitions}\label{sec_operations}
	
	In this subsection, unless otherwise stated, $\lambda$ and $\gamma$ are partitions whose multiplicities are given by the sequence $(a_i)$ and $(b_i)$ respectively. We denote the largest part of $\lambda$ by $k_\lambda$ and the largest part of $\gamma$ by $k_\gamma$.  We define the \emph{sum} $\lambda \oplus \gamma$ as the partition with multiplicities $a_i+b_i$ and largest part $\max\{k_\lambda, k_\gamma\}$.  In other words, we obtain the sum $\lambda\oplus\gamma$ by combining all of the parts of $\lambda$ and $\gamma$ into one combined partition.  This operation $\oplus$ on partitions has been previously defined \cite{A}, and Schneider more recently named this operation the \emph{product} of two partitions as part of his multiplicative partition theory \cite[Def.\ 1.2.2]{S}.
	
	The normalized partition polynomial of $\lambda \oplus \gamma$ is given by
	\begin{equation*}
	\hat{f}_{\lambda \oplus \gamma}(x) = \frac{1}{\ell(\lambda) + \ell(\gamma)} \left( \sum_{i=1}^{k_\lambda} a_ix^i + \sum_{i=1}^{k_\gamma} b_ix^i \right).
	\end{equation*}
	The following proposition shows what happens when we integrate the normalized partition polynomial of the sum of two partitions.
	
	\begin{proposition}
		\label{addition_formula}
		We have that
		\begin{align*}
		\int_{0}^{1}\hat{f}_{\lambda \oplus \gamma}(x)\,dx = \frac{\ell(\lambda)}{\ell(\lambda)+\ell(\gamma)} \int_{0}^{1} \hat{f}_\lambda(x)\,dx + \frac{\ell(\gamma)}{\ell(\lambda)+\ell(\gamma)} \int_{0}^{1} \hat{f}_\gamma(x)\,dx .
		\end{align*} 
	\end{proposition}
	
	\begin{proof}
		We proceed by direct computation.
		\begin{align*}
		\int_{0}^{1}\hat{f}_{\lambda \oplus \gamma}(x)\,dx & = \frac{1}{\ell(\lambda) + \ell(\gamma)} \left( \sum_{i=1}^{k_1} \frac{a_i}{i+1} + \sum_{i=1}^{k_2} \frac{b_i}{i+1} \right) \\
		& = \frac{1}{\ell(\lambda) + \ell(\gamma)} \left( \ell(\lambda) \frac{1}{\ell(\lambda)} \sum_{i=1}^{k_1}\frac{a_i}{i+1} + \ell(\gamma) \frac{1}{\ell(\gamma)}\sum_{i=1}^{k_2} \frac{b_i}{i+1} \right) \\
		& = \frac{1}{\ell(\lambda) + \ell(\gamma)} \left( \ell(\lambda) \int_{0}^{1}\hat{f}_\lambda(x)\,dx + \ell(\gamma) \int_{0}^{1}\hat{f}_\gamma(x)\,dx \right).
		\end{align*}
		The result follows.
	\end{proof}
	
	A few important corollaries result from this proposition.
	
	\begin{corollary}\label{cor1}
		We have that
		\begin{equation*}
		\int_{0}^{1} \hat{f}_{\lambda \oplus \lambda}(x)\,dx = \int_{0}^{1} \hat{f}_{\lambda}(x)\,dx .
		\end{equation*}
	\end{corollary}
	
	\begin{proof}
		By Proposition \ref{addition_formula}, we have that
		\begin{align*}
		\int_{0}^{1} \hat{f}_{\lambda \oplus \lambda}(x)\,dx & = \frac{\ell(\lambda)}{\ell(\lambda)+\ell(\lambda)} \int_{0}^{1} \hat{f}_\lambda(x)\,dx + \frac{\ell(\lambda)}{\ell(\lambda)+\ell(\lambda)} \int_{0}^{1} \hat{f}_\lambda(x)\,dx \\
		& = \frac{1}{2} \int_{0}^{1} \hat{f}_\lambda(x)\,dx + \frac{1}{2} \int_{0}^{1} \hat{f}_\lambda(x)\,dx \\
		& = \int_{0}^{1}\hat{f}_\lambda(x)\,dx .
		\end{align*}
	\end{proof}
	
	\begin{corollary}\label{cor2}
		The sum of two partitions of the same length has integral equal to the average of the individual integrals. Explicitly, if $\ell(\lambda) = \ell(\gamma)$, then we have that
		\begin{equation*}
		\int_{0}^{1}\hat{f}_{\lambda \oplus \gamma}(x)\,dx = \frac{1}{2} \left(\int_{0}^{1}\hat{f}_\lambda(x)\,dx + \int_{0}^{1}\hat{f}_\gamma(x)\,dx  \right) .
		\end{equation*}
	\end{corollary}
	
	The proof of Corollary \ref{cor2} is similar to the proof of Corollary \ref{cor1}.  These results will be used to prove Theorem \ref{main_result} in Section \ref{sec_integralproofs} through a binary search type argument.

\section{Proof of Theorem \ref{main_result}}\label{sec_integralproofs}

In order to prove Theorem \ref{main_result}, it suffices to fix an arbitrary $c \in [0, 1/2]$ and to show that there exists a sequence of partitions $(\delta_s)$ such that $$\lim_{s \to \infty} \int_{0}^{1} \hat{f}_{\delta_s}(x)\,dx = c.$$  We first provide constructive proofs of the $c=0$ and $c=1/2$ cases of Theorem \ref{main_result}.
	
	\begin{lemma}
		\label{lemma_edge_50}
		There exists a sequence of non-trivial partitions $(\alpha_s)$ such that
		\begin{equation*}
		\lim\limits_{s \to \infty} \int_{0}^{1}\hat{f}_{\alpha_s}(x)\,dx = 0 .
		\end{equation*}
	\end{lemma}
	
	\begin{proof}
		Define the partition $\alpha_s$ by $\langle 1^1,2^0,3^0,\dots,s^{s-1} \rangle$. Then we observe that $\hat{f}_{\alpha_s}(x) = \frac{1}{s} (x+(s-1)x^s )$, and
		\begin{equation*}
		\lim\limits_{s \to \infty} \int_{0}^{1}\hat{f}_{\alpha_s}(x)\,dx = \lim\limits_{s \to \infty} \frac{1}{s}\left(\frac{1}{2}+\frac{s-1}{s+1} \right) = 0 .
		\end{equation*}
	\end{proof}
A similar result holds for non-trivial partitions whose integral approaches $1/2$. 
	
	\begin{lemma}
		\label{lemma_edge_0}
		There exists a sequence of non-trivial partitions $(\beta_s)$ such that
		\begin{equation*}
		\lim\limits_{s \to \infty} \int_{0}^{1}\hat{f}_{\beta_s}(x)\,dx = 1/2 .
		\end{equation*}
	\end{lemma}
	
	\begin{proof}
		Define the partition $\beta_s$ by $\langle 1^{s-1},2^0,3^0,\dots,(s-1)^0,s^1 \rangle$. Then we have that $\hat{f}_{\beta_s}(x) = \frac{1}{s}((s-1)x+x^s)$, and
		\begin{equation*}
		\lim\limits_{s \to \infty} \int_{0}^{1}\hat{f}_{\beta_s}(x)\,dx = \lim\limits_{s \to \infty} \frac{1}{s}\left(\frac{s-1}{2}+\frac{1}{s+1}\right) = 1/2 .
		\end{equation*}
	\end{proof}
	The following theorem, combined with Lemmas \ref{lemma_edge_50} and \ref{lemma_edge_0}, will give Theorem \ref{main_result} immediately.
	
	\begin{theorem}
		\label{main_lemma}
		Suppose that $\alpha$ and $\beta$ are partitions such that $$\int_{0}^{1}\hat{f}_\alpha(x)\,dx = a\hspace{.5cm}\text{and}\hspace{.5cm}\int_{0}^{1} \hat{f}_\beta(x)\,dx = b.$$ If $0 < a < b < 1/2$ and $\ell(\alpha) = \ell(\beta)$, then for every $c \in (a, b)$, there exists a sequence $(\delta_s)$ of partitions such that
		
		\begin{equation*}
		\lim\limits_{s \to \infty} \int_{0}^{1} \hat{f}_{\delta_s}(x)\,dx = c .
		\end{equation*}
	\end{theorem}
	
	\begin{proof}
		Let $\epsilon > 0$ be given. We construct $(\delta_s)$ recursively by first constructing two sequences $(\beta_s)$ and $(\alpha_s)$ of partitions such that each $\beta_s$ approximates $c$ from above and each $\alpha_s$ approximates $c$ from below. Define $\beta_1 = \beta$, $\alpha_1 = \alpha$, and $\delta_1 = \alpha_1 \oplus \beta_1$. The sequence $(\delta_s)$ will be defined by $\delta_s = \alpha_s \oplus \beta_s$. 
		
		Let $r \geq 1$. If $c = {\displaystyle\int_{0}^{1}} \hat{f}_{\delta_r}(x)\,dx$, then define $\delta_s = \delta_r$. Otherwise, we have one of the following two cases.\\
		\noindent\textbf{Case 1:} $c < {\displaystyle\int_{0}^{1}} \hat{f}_{\delta_r}\,dx$. Define $\beta_{r+1} = \delta_r$ and $\alpha_{r+1} = \alpha_r \oplus \alpha_r$. Since $\ell(\beta_{r+1}) = 2^{r+1}\ell(\alpha) = \ell(\alpha_{r+1})$, we see that
		\begin{equation*}
		\int_{0}^{1}\hat{f}_{\delta_{r+1}}(x)\,dx = \frac{1}{2} \left(\int_{0}^{1}\hat{f}_{\beta_{r+1}}(x)\,dx + \int_{0}^{1}\hat{f}_{\alpha_{r+1}}(x)\,dx \right).
		\end{equation*}
		
		\noindent\textbf{Case 2:} $c > {\displaystyle\int_{0}^{1}} \hat{f}_{\delta_r}\,dx$. Define $\beta_{r+1} = \beta_r \oplus \beta_r$ and $\alpha_{r+1} = \delta_r$. Since $\ell(\beta_{r+1}) = 2^{r+1}\ell(\alpha) = \ell(\alpha_{r+1})$, we see that
		\begin{equation*}
		\int_{0}^{1}\hat{f}_{\delta_{r+1}}(x)\,dx =   \frac{1}{2}\left(\int_{0}^{1}\hat{f}_{\beta_{r+1}}(x)\,dx + \int_{0}^{1}\hat{f}_{\alpha_{r+1}}(x)\,dx \right) .
		\end{equation*}
		Note that by definition,
		\begin{equation*}
		\int_{0}^{1}\hat{f}_{\delta_s}(x)\,dx = \frac{1}{2} \left(\int_{0}^{1}\hat{f}_{\beta_s}(x)\,dx + \int_{0}^{1}\hat{f}_{\alpha_s}(x)\,dx \right) .
		\end{equation*}
		Since the integral for $\alpha_s$ is less than the integral for $\beta_s$, we obtain the following inequality:
		\begin{equation*}
		\int_{0}^{1}\hat{f}_{\alpha_s}(x)\,dx < \int_{0}^{1}\hat{f}_{\delta_s}(x)\,dx < \int_{0}^{1}\hat{f}_{\beta_s}(x)\,dx.
		\end{equation*}
		
		Now let $N > \log_2\left(\frac{b-a}{\epsilon} \right)$. The previous inequality, along with the fact that $c$ is an upper bound for the integral for $\alpha_s$, yields the following for all $s \geq N$:
		\begin{align*}
		\left| \int_{0}^{1}\hat{f}_{\delta_s}(x)\,dx - c \right| &
		< \int_{0}^{1}\hat{f}_{\beta_s}(x)\,dx - \int_{0}^{1}\hat{f}_{\alpha_s}(x)\,dx \leq \frac{b-a}{2^s} \leq \frac{b-a}{2^N} < \epsilon.
		\end{align*}
		Therefore, $\lim\limits_{s \to \infty} {\displaystyle\int_{0}^{1}}\hat{f}_{\delta_s}(x)\,dx = c$ as desired. 
	\end{proof}
	
	We are now ready to prove Theorem \ref{main_result}.
	
	\begin{proof}[Proof of Theorem \ref{main_result}]
		The edge cases $c = 0$ and $c = 1/2$ follow from Lemma \ref{lemma_edge_50} and Lemma \ref{lemma_edge_0}, respectively. Thus, it suffices to show that the result holds for $c \in (0, 1/2)$. Define the partitions
		\begin{align*}
		\alpha_s = \langle 1^1,2^0,3^0,\dots,(s-1)^0,s^{s-1} \rangle\hspace{.5cm}\text{and}\hspace{.5cm}\beta_s = \langle 1^{s-1},2^0,3^0,\dots,(s-1)^0,s^1 \rangle.
		\end{align*}
		Then we have that
		\begin{align*}
		\int_{0}^{1}\hat{f}_{\alpha_s}(x)\,dx & = \frac{1}{s}\left(\frac{1}{2}+\frac{s-1}{s+1} \right), \\
		\int_{0}^{1}\hat{f}_{\beta_s}(x)\,dx & = \frac{1}{s}\left(\frac{s-1}{2}+\frac{1}{s+1} \right).
		\end{align*}
		Since $\ell(\alpha_s) = \ell(\beta_s)$ for all $s \in \mathbb{N}$, Theorem \ref{main_lemma} shows that the result holds for all $c$ with
		\begin{equation*}
		\int_{0}^{1}\hat{f}_{\alpha_s}(x)\,dx < c < \int_{0}^{1}\hat{f}_{\beta_s}(x)\,dx .
		\end{equation*}
		For increasing $s$, these are nested intervals. The union of all such intervals is the set $(0, 1/2)$. This proves Theorem \ref{main_result}.
	\end{proof}

\section{Proofs of Theorems \ref{avg1} and \ref{avg2}}\label{sec_averages}

\begin{proof}[Proof of Theorem \ref{avg1}]
    Recall that one may compute $\mathrm{Avg}(n, \ell)$ by evaluating the integral of the sum of all partitions of $n$ with $\ell$ parts.  It is straightforward to enumerate all partitions of $n$ into one part and all partitions of $n$ into two parts. Adding all of the parts together for $\ell=1$ and $\ell=2$ gives the partition $\lambda_1 = \langle 1^0,2^0,\dots,n^1 \rangle$ for all $n$, and the partitions $$\lambda_2^\text{odd} = \langle 1^1,2^1,\dots,(n-1)^1,n^0 \rangle$$ if $n$ is odd or $$\lambda_2^\text{even}=\langle1^1,2^1,\dots,((n/2)-1)^1,(n/2)^2,((n/2)+1)^1,\dots,(n-1)^1,n^0\rangle$$ if $n$ is even.  Let $H_n$ denote the $n$th Harmonic number.  Then the averages are $$\mathrm{Avg}(n, 1) = \int_{0}^{1} \hat{f}_{\lambda_1}(x)\,dx = \frac{1}{n+1}$$ for all $n$; and 
    \begin{align*}
        \mathrm{Avg}(n, 2) & = \int_{0}^{1} \hat{f}_{\lambda_2^\text{odd}}(x)\,dx = \frac{1}{2 \cdot \left\lfloor\frac{n}{2}\right\rfloor}\left(\frac{1}{2} + \frac{1}{3} + \cdots + \frac{1}{n} \right) = \frac{1}{2\cdot \left\lfloor\frac{n}{2}\right\rfloor} (H_n - 1 )
    \end{align*}
    if $n$ is odd, or
    \begin{align*}
        \mathrm{Avg}(n, 2) & = \int_{0}^{1} \hat{f}_{\lambda_2^\text{even}}(x)\,dx = \frac{1}{2 \cdot \left\lfloor\frac{n}{2}\right\rfloor}\left(\frac{1}{2} + \frac{1}{3} + \cdots + \frac{1}{n} + \frac{1}{n/2}\right) = \frac{1}{2\cdot \left\lfloor\frac{n}{2}\right\rfloor} \left(H_n - 1 +\frac{2}{n}\right)
    \end{align*}
    if $n$ is even.  Noting that $\left\lfloor\frac{n}{2}\right\rfloor \leq \frac{n}{2}$, we see that
    \begin{align*}
        \mathrm{Avg}(n, 2) & \geq \frac{1}{2 \cdot \frac{n}{2}} (H_n - 1 )\geq \frac{1}{n+1} (H_n - 1 )\geq \mathrm{Avg}(n, 1)
    \end{align*}
    for all $n\geq4$.  One can check by hand that the result also holds for $n<4$.  Thus, we have that $\mathrm{Avg}(n, 1) \leq \mathrm{Avg}(n, 2)$ for all $n \in \mathbb{N}$.
\end{proof}

Note in the above proof that we have the asymptotic expression $\mathrm{Avg}(n, 2)\sim\ln(n)/n$ as $n\rightarrow\infty$, since the $n$th harmonic number is asymptotic to $\ln(n)$.
    
    \begin{proof}[Proof of Theorem \ref{avg2}]
    We will show that $\mathrm{Avg}(n, 3)$ is asymptotically bounded below by $2\ln(n)/n$. Let $n \in \mathbb{N}$, and let $\lambda_3$ be the sum of all partitions of $n$ into $3$ parts. To get a lower bound on the multiplicity of each part in $\lambda_3$, let $1 \leq i \leq n-2$, and note that the following sums give partitions of $n$ into $3$ parts:
    \begin{align*}
        & 1+i+(n-(i+1)) = 2+i+(n-(i+2)) = \cdots = (i-1)+i+(n-(2i-1)), \\
        & i+i+(n-2i) = i+(i+1)+(n-(2i+1)) =  \cdots = i+\left\lfloor\frac{n-i}{2}\right\rfloor+\left\lceil\frac{n-i}{2}\right\rceil.
    \end{align*}
    Counting the number of times $i$ appears, there are at least $i-1$ instances in the first row and at least $\left\lfloor\frac{n-i}{2}\right\rfloor-i+1$ in the second row. This gives a lower bound of
    \begin{align*}
        i-1+\left\lfloor\frac{n-i}{2}\right\rfloor-i+1 & = \left\lfloor\frac{n-i}{2}\right\rfloor \geq \frac{n-i}{2}-1
    \end{align*}
    parts of size $i$ among all partitions of $n$ into 3 parts.  Define the function $g : [0, n-2] \rightarrow \mathbb{R}$ by $g(x) = \frac{n}{2}-1 - \frac{n/2}{n-2} x$. We use this function to obtain a lower bound for the integral of $f_{\lambda_3}$, as it always lies below the bound achieved above.  We see that
    \begin{align*}
        \mathrm{Avg}(n, 3) & = \int_{0}^{1} \hat{f}_{\lambda_3}(x)\,dx \\
        & \geq \frac{1}{\ell(\lambda_3)} \int_{0}^{n-2} \frac{g(x)}{x+1} \,dx \\
        & = \frac{1}{\ell(\lambda_3)} \cdot \left( \left( \frac{n}{2} - 1\right)\ln(n-1) - ( (n-2)-\ln(n-1)) \cdot \frac{n/2}{n-2} \right) \\
        & \sim \frac{1}{\ell(\lambda_3)} \cdot \frac{n}{2} \ln(n)
    \end{align*}
    as $n\rightarrow\infty$.  Also, as $n\rightarrow\infty$ we have that $\ell(\lambda_3) \sim 3\cdot\frac{n^2}{12}$, by a known asymptotic formula for the number of partitions of size $n$ into 3 parts (see \cite{KK}). This yields
    \begin{align*}
        \mathrm{Avg}(n, 3) & \geq \frac{1}{\ell(\lambda_3)} \cdot \frac{n}{2}\ln(n)\sim \frac{1}{3\cdot\frac{n^2}{12}} \cdot \frac{n}{2}\ln(n)= \frac{2\ln(n)}{n}\sim 2 \mathrm{Avg}(n, 2)
    \end{align*}
    as $n\rightarrow\infty$.  Therefore, $\mathrm{Avg}(n,3)\geq\mathrm{Avg}(n,2)$ for sufficiently large $n$, as desired.
\end{proof}

\section{Acknowledgments}
The authors would like to thank Matthew Just and Robert Schneider for many enlightening discussions regarding partition statistics, Ken Ono for helpful comments on the content of this paper, and Dario Mathi\"a for investigating a preliminary version of Question \ref{question2} and sharing insightful observations and calculations.  The authors would also like to thank the referees for their suggestions which greatly improved this paper.  The first author is supported by an AMS-Simons Travel Grant and an internal grant from the University of Texas at Tyler.

\noindent 
2020 \emph{Mathematics Subject Classification}: Primary 11P82. 
Secondary 05A17, 11C08.

\noindent 
\emph{Key words and phrases}: Partition, polynomial. 

\noindent 
(Concerned with sequences
\seqnum{A000041},
\seqnum{A008277}, and
\seqnum{A305078}.)

\end{document}